\documentclass[letterpaper,11pt]{amsart}
\usepackage{amsfonts}

\usepackage{amsthm}
\usepackage{enumerate}
\usepackage{mathrsfs}
\usepackage{cancel}
\usepackage{amscd, amssymb}
\usepackage{amsmath,amscd}
\usepackage{enumerate}
\usepackage{comment}
\usepackage{graphicx}
\usepackage[usenames,dvipsnames]{color}
\usepackage{bm}
\usepackage{url}

\newtheorem{dummy}{}[section]
\newtheorem{thm}[dummy]{Theorem}
\newtheorem{defn}[dummy]{Definition}

\newtheorem{lem}[dummy]{Lemma}
\newtheorem{cor}[dummy]{Corollary}
\newtheorem{rem}[dummy]{Remark}

\newcommand{\rr}{\mathbb{R}}

\newcommand{\ee}{\mathbb{E}}

\newcommand{\cc}{\mathbb{C}}

\newcommand{\zz}{\mathbb{Z}}

\newcommand{\pt}{\mathcal{PRT}}

\newcommand{\unlabt}{\Lambda}
\newcommand{\labt}{\Gamma}
\newcommand{\cont}{\text{Cont}}

\begin{document}

\title{Exact Maximum-Entropy Estimation with Feynman Diagrams
}


%

\author{Amitai Netser Zernik}
\address{Institute for Advanced Study, Princeton, New Jersey, USA}
\email{azernik@ias.edu}

\author{Tomer M. Schlank}
\address {Department of Mathematics, The Hebrew University of Jerusalem, Jerusalem, Israel}
\email{tomer.schlank@gmail.com}
\author{Ran J. Tessler}
\address{R. J. Tessler:\newline Institute for Theoretical Studies, ETH Zurich, Switzerland}
\email{ran.tessler@eth-its.ethz.ch}

%
%
%
%


\maketitle

\begin{abstract}
A longstanding open problem in statistics is finding an explicit expression for the probability measure which maximizes entropy with respect to given constraints. In this paper a solution to this problem is found, using perturbative Feynman calculus. The explicit expression is given as a sum over weighted trees.
\end{abstract}

\section{Introduction}

Given a finite set $\Sigma$,  the relationship between distributions on $\Sigma$ and an observable $r:\Sigma \to \mathbb{R}$ on $\Sigma$ is a two way street. First, given a distribution $P_0$ on $\Sigma$ we can ask for the expectation of $r,~\mathbb{E}_{P_0} r.$

In the other direction, suppose the distribution $P_0$ is unknown but we are given, for $1\leq i\leq k$, the expectation $\rho_i$ of some observable
$r_i:\Sigma \to \mathbb{R}$ with respect to $P_0$. Of course, in general there will be infinitely many distributions $P$ such that \begin{equation}\label{eq:constraints}\mathbb{E}_{P}r_i = \rho_i ,\quad \forall 1 \leq i \leq k. \end{equation}
Arguably, among these distributions $P$ the one that maximizes the entropy is the one that best reflects the information given by the expectations. This approach to estimation was first expounded by E. T. Jaynes in two papers, \cite{JaynesI,JaynesII}. The introduction chapter in \cite{Frisch} provides a nice review of the subject, another comprehensive source is \cite{Cover}.

A classical theorem of Ludwig E. Boltzmann shows that the maximum entropy distribution $P$ belongs to a finite-dimensional exponential family of distributions, or \emph{Boltzmann-Gibbs distributions}, parameterized by $\lambda_i,~0\leq i\leq k$. It is not hard to prove that these parameters are analytic functions of $\{\rho_i\}_{i=1}^k.$
It is possible to directly compute the first few orders of the series expansion of each $\{\lambda_i\}_{i=0}^k$ in terms of $\{\rho_i\}_{i=1}^k$. For example, the quadratic approximation defines the \emph{linear regression} multivariate normal distribution. As one tries to calculate higher orders, however, the computations quickly get out of hand. Many algorithms have been proposed for approximating the distribution numerically or by other means.

In this paper we give an explicit, combinatorial formula for computing the full Taylor expansion of $\lambda_i(\rho_1,\ldots,\rho_k)$ in terms of the joint moments of the observables $r_i$ with respect to the uniform distribution. An alternative formula in terms of cumulants is also proven, which is computationally more efficient.

In estimation problems, one is often not interested in the distribution $P$ itself, but rather in the expectation value $\sigma := \mathbb{E}_Ps$ of some observable $s:\Sigma \to \rr$. In Theorems \ref{thm:cumulants_M} and \ref{thm:7} we compute the Taylor expansion of $\sigma(\rho_1,\ldots,\rho_k).$

{\sloppy}It is worth mentioning that given the higher moments of $\{r_i\},$ (or $\{r_i\}\cup \{s\}$) the computation is completely independent of the size of the alphabet $\Sigma$. More precisely, to compute the Taylor expansion to any given order $d$ only the joint moments of order $\leq d+1$ are required.


The main tool for proving these results is a formula for the Taylor expansion of a perturbed critical point. Folk theorems along these lines go back at least to Richard P. Feynman, and constitute the "classical", or "tree level" part of what is generally called the Feynman Calculus. We found the exposition in \cite{Etingof} very useful. At any rate, the discussion in Section \ref{sec:feynman} is completely self-contained, and readers may find it applicable to other optimization problems.

Some application of the problem we consider, in its relative version to be defined below are \cite{Shell} in chemical physics, \cite{Avellanda} in mathematical finance and the recent \cite{Lin}, which in page 6, argues that part of the success of deep learning techniques is a consequence of the structure of probability distributions that maximizes the entropy. These references are of course just a drop in the sea of applications.
We therefore believe that the formulas presented here can find applications in many estimation and classification problems, different areas of science, and have considerable theoretical value as well.




In the next subsection we define precisely the problem we would like to solve, and discuss some analytic properties of the solutions.
In the next section we state our main results. In section \ref{sec:feynman} we develop the main technical tool, Feynman calculus, and discuss its application to finding critical points of series expansions. This is used to prove the main results in Section \ref{sec:proofs}.

%

\subsection{Kullback-Liebler constraint problems}
The problem of maximizing the entropy of a distribution $P$ on an alphabet $\Sigma$ subject to some constraints is an instance of the somewhat more general problem of \emph{minimizing} the \emph{Kullback-Liebler divergence}, or \emph{relative entropy} $D(P||Q)$ of $P$ relative to a given reference distribution $Q$. Indeed, if $Q$ is the uniform distribution then $D(P||Q) = \log |\Sigma| - H(P)$ where $H(P)$ is the entropy of $P$. With this in mind, we can state our results as follows.

\begin{defn}
\emph{A Kullback-Liebler constraint problem} (or KL constraint
problem) consists of a triple $\left(\Sigma,Q,\left\{ r_{i}\right\} _{i=1}^{k}\right)$ where
\begin{enumerate}
\item $\Sigma$ is a finite set called \emph{the alphabet.}
\item $Q=\left\{ q_{\sigma}\right\} $ is a probability distribution on $\Sigma,$
called \emph{the reference probability distribution.}
\item $r_{i}:\Sigma\to\rr$ for $1\leq i\leq k$ are functions called the \emph{constraint functions.}
\end{enumerate}
We assume that $q_\sigma > 0$ for all $\sigma \in \Sigma$ and that the $r_i$ are affinely independent as elements of the vector space $\rr^\Sigma.$ By that we mean that no non trivial linear combination of the functions $r_i$ and the constant function $1$ yields the $0$ function. One can reduce KL constraint problems which do not satisfy these conditions to ones that do, in an obvious way.

A KL constraint problem $\left(\Sigma,Q,\left\{ r_{i}\right\} _{i=1}^{k}\right)$ will be called \emph{augmented} if we are given an additional \emph{target function} $s:\Sigma\to\rr$. 

A KL constraint problem will be called \emph{normalized
}if $\mathbb{E}_{Q}r_{i}=0$ for $1\leq i \leq k$ and $\mathbb{E}_{Q}r_{i}r_{j}=\delta_{ij}$ for $1\leq i,j\leq k$. An augmented KL constraint problem will be called normalized if the non augmented problem is normalized and, in addition, $\mathbb{E}_{Q}s=0,~\mathbb{E}_{Q}sr_{i}=0,$ for all $1\leq i\leq k.$
\end{defn}
\begin{rem}\label{rem:normalizng}
There is no loss of generality in assuming that a KL problem is normalized. More precisely, given a KL constraint problem $(\Sigma,Q,\{r_i\}_{i=1}^k),$ one obtains a normalized problem in two steps.\\
(i) \emph{Replace $r_i$ by $r_i':=r_i - \mathbb{E}(r_i)$, so $\mathbb{E}_Q r_i'=0$.} \\
Let $E_0 \subset \rr^\Sigma$ denote the sub-vector space $\{f\in\rr^\Sigma|\mathbb{E}_Q{f}=0\}$. Because $q_\sigma>0$ the covariance restricts to a non-degenerate positive definite pairing on $E_0$:
\begin{equation}
Cov(f,f') = \sum_\sigma q_\sigma \cdot f(\sigma)\cdot f'(\sigma).
\end{equation}
(ii) \emph{Use the Gram-Schmidt process to replace the sequence $\{r_i'\}_{i=1}^k$ by a $Cov$-orthonormal sequence $\{r_i''\}_{i=1}^k$ which has the same linear span.}

$(\Sigma,Q,\{r_i''\}_{i=1}^k)$ defines an equivalent KL-constraint problem, in the sense that the set of distributions $P$ satisfying Equation (\ref{eq:constraints}) is equal to the set of distributions $P$ satisfying
\begin{equation}
\mathbb{E}_Q r_i'' = \rho''_i
\end{equation}
where $(\rho_i'')_{i=1}^k$ is obtained from $(\rho_i)_{i=1}^k$ by the affine-linear transformation effecting the change from $r_i$ to $r_i''$ computed above.

Similar reasoning shows that there is no loss of generality in assuming that an augmented KL constraint problem is normalized.
\end{rem}

Given a normalized KL constraint problem $\left(\Sigma,Q,\left\{ r_{i}\right\} _{i=1}^{k}\right)$
and sufficiently small parameters $\rho_{i}\in\rr$, $1\leq i\leq k$,
we consider the probability distribution $P=\left\{ p_{\sigma}\right\} _{\sigma\in\Sigma}$
minimizing the Kullback-Liebler divergence of $P$ relative to $Q$,
\begin{equation}\label{eq:d_KL}
D_{KL}\left(P||Q\right)=\mathbb{E}_{P}\log\left(\frac{P}{Q}\right)=\sum_{\sigma\in\Sigma}p_{\sigma}\log\left(\frac{p_{\sigma}}{q_{\sigma}}\right)
\end{equation}
subject to the constraints $\mathbb{E}_{P}r_{i}=\rho_{i}$.

By Lagrange multipliers (see the beginning of Subsection \ref{sec:exp_param}) we have
\begin{equation}
\label{eq:Lagrange}
p_{\sigma}=q_{\sigma}\exp\left(-1-\lambda_{0}-\sum_{i=1}^{k}\lambda_{i}r_{i}\left(\sigma\right)\right)
\end{equation}
for some numbers $\lambda_{i}=\lambda_{i}\left(\rho_{1},\ldots,\rho_{k}\right)\in \rr$ which we call \emph{the exponential parameters},

In Appendix \ref{sub:pf of analyticity} we prove
\begin{lem}\label{lem:analytic}
For $1\leq i \leq k$ the exponential parameter $\lambda_i$ is an analytic function of $\rho_1,\ldots,\rho_k$.
\end{lem}

Now suppose the normalized KL-constraint problem is augmented by a target function $s:\Sigma \to \rr$. We call $\sigma = \sigma(\rho_1,\ldots,\rho_k) = \ee_Ps$ the \emph{target expectation}.

Lemma \ref{lem:analytic} has the following corollary,
\begin{cor}
The target expectation $\sigma(\rho_1,\ldots,\rho_k)$ is an analytic function of $\rho_1,\ldots,\rho_k$
\end{cor}
\begin{proof}
We have
\[
\sigma(\rho_1,\ldots,\rho_k)=\ee_{P}s=\mu\left(\lambda_{1}(\rho_1,\ldots,\rho_k),\ldots,\lambda_{k}(\rho_1,\ldots,\rho_k)\right)
\]
for
\[
\mu(\lambda_{1},\dots\lambda_{k})=\frac{\sum_{\sigma}s(\sigma)q_{\sigma}\exp\left(-\sum_{i=1}^{k}\lambda_{i}r_{i}\left(\sigma\right)\right)}{\sum_{\sigma}q_{\sigma}\exp\left(-\sum_{i=1}^{k}\lambda_{i}r_{i}\left(\sigma\right)\right)}
\]
which is clearly analytic in a neighborhood of $\lambda_{1}=\cdots=\lambda_{k}=0.$. Since the composition of analytic functions is analytic, the result follows.
\end{proof}

{\sloppy}The upshot is that the Taylor expansions of $\lambda_i(\rho_1,\ldots,\rho_k)$, $1\leq i\leq k$ and of $\sigma(\rho_{1},\dots,\rho_{k})$ converge in an open polydisc around ${\rho_{1}=\cdots=\rho_{k}=0}$. We refer
to the power series expansions of these functions about the origin
as \emph{the perturbed exponential parameters}
\begin{equation}\label{eq:lambda_i}
\lambda_{i}\left(\rho_{1},\ldots,\rho_{k}\right)=\sum_{I}L_{i,I}\rho^{I}
\end{equation}
and \emph{the perturbed expectation function}
\begin{equation}\label{eq:mu}
\sigma\left(\rho_{1},\ldots,\rho_{k}\right)=\sum_{I}M_{I}\rho^{I}
\end{equation}
here $I=\left(i_{1},\ldots,i_{k}\right)$ ranges over $\zz_{\geq 0}^k$ and $\rho^{I}:=\prod_{j=1}^{k}\rho_{j}^{i_{j}}$.
%
\section{Results}
\subsection{Moment trees}\label{subsec:moment-trees}
Let $\left(\Sigma,Q,\left\{ r_{i}\right\} _{i=1}^{k}\right)$ be
a normalized KL constraint problem. It will be notationally convenient
to write $r_{0}:\Sigma\to\mathbb{R}$ for the constant function $1$. If $\left(\Sigma,Q,\left\{ r_{i}\right\} _{i=1}^{k},s\right)$ is an \emph{augmented} KL constraint problem we introduce an additional function $r_{k+1}:\Sigma\to\mathbb{R}$ which we set equal to the target function $s$.
\begin{defn}
A \emph{rooted tree} (or RT) $\labt$ is a tree with vertices $V=V(\labt)$ and a nonempty set of edges $E=E(\labt),$ together with a distinguished leaf $v_{out} \in V$.
We call the vertex to which the distinguished leaf $v_{out}$ is connected \emph{the root} of $\labt.$ 
We define the \emph{internal vertices} $V_{in}$ to be the vertices which are not leaves. In particular, when $|V(\labt)|=2,~V_{in}=\emptyset.$
\end{defn}

\begin{defn}
\label{def:moment-tree}
For $j=0,\ldots,k+1$ a \emph{$j$-moment tree} $\labt$ for the normalized KL
constraint problem $\left(\Sigma,Q,\left\{ r_{i}\right\} _{i=1}^{k}\right)$
is a rooted tree whose edges are labeled by $\{r_{i}\}_{i=0}^{k+1}$
subject to the following conditions:
\begin{itemize}
\item Shape conditions: each vertex which is not a leaf has ${\mbox{degree }\geq3.}$
\item Labeling conditions: The edge connected to $v_{out}$ is labeled by
$r_{j};$ the edges connected to the other leaves are labeled by $r_{1},\ldots,r_{k}$.
All other edges are labeled by $r_{0},\ldots,r_{k}$.
\end{itemize}
A \emph{moment tree} is a $j-$moment tree for some $0\leq j\leq k+1$.

Two moment trees $\Gamma_1,\Gamma_2,$ are \emph{isomorphic} if there exist bijections $f_V:V(\Gamma_1)\to V(\Gamma_2),f_E:E(\Gamma_1)\to E(\Gamma_2)$ such that
\begin{enumerate}
\item $f_V(v_{out}(\Gamma_1)) = v_{out}(\Gamma_2).$
\item $f_E$ maps the edges of $v$ to the edges of $f_V(v).$
\item The label of $f_E(e)$ agrees with the label of $e.$
\end{enumerate}
An \emph{automorphism} of $\Gamma$ is an isomorphism from $\Gamma$ to itself. The equivalence classes of moment trees modulo isomorphisms are called the \emph{isomorphism types}.
\end{defn}
To a moment tree $\labt$ we associate a \emph{leaf multi-index $I_{\labt}=(a_{1},\ldots,a_{k})$}
with $a_{i}$ the number of leaves with edges labeled by $r_{i},$ excluding $v_{out}$ (but including the root, in case it is also a leaf).
We define the \emph{order} of $\Gamma$ to be $|I_{\labt}| = \sum_{i=1}^k a_i.$
Note that since the degree of inner nodes is $\geq3$ there are only
a finite number of moment trees with any given leaf multi-index.
\begin{defn}\label{def:moment_tree_amp}
Let $\labt$ be a moment tree. Let $v$ be an inner vertex of $\labt$ of degree $d$ with edges labeled by $r_{i_{1}},\ldots,r_{i_{d}}$.
\begin{enumerate}
\item The \emph{coupling
value} $C_{v}$ of $v$ is
\begin{equation}\label{eq:moment_couplings}
C_{v}=(-1)^{d+1}\mathbb{E}_{Q}r_{i_{1}}\cdots r_{i_{d}}.
\end{equation}
\item
The \emph{amplitude }$A_{\labt}$ of $\labt$
is
\[
A_{\labt}=(-1)^{|I_\Gamma|}\prod_{v\in V_{in}(\labt)}C_{v},
\]
or $-1$ when $V_{in}=\emptyset.$
\end{enumerate}
\end{defn}
\begin{thm}\label{thm:taylor-exp-parameters} Given a normalized KL constraint problem, the coefficients of the perturbed exponential parameters of Equation \ref{eq:lambda_i} are given by
\begin{equation}
\label{eq:params-by-moments}
L_{j,I}=\sum_{\labt} \frac{A_{\labt}}{|Aut(\labt)|}
\end{equation}
for $1\leq j \leq k$ where the sum runs over representatives
$\labt$ for all isomorphism types of $j$-moment trees $\labt$ with $I_{\labt}=I,$ and where $Aut(\labt)$ is the group of automorphisms of $\labt$, i.e.
the group of rooted tree automorphisms, preserving the labels and the root.
\end{thm}

\begin{thm}\label{thm:7}
Given an augmented normalized KL constraint problem, the perturbed expectation coefficients of Equation \ref{eq:mu} are given by
\[M_{I}=\sum_{\labt} \frac{A_{\labt}}{|Aut(\labt)|}\]
where the sum runs over representatives $\labt$ for all isomorphism types of $\left(k+1\right)$-moment
trees $\labt$ with $I_{\labt}=I$.
\end{thm}

\begin{rem}
Suppose we are given a normalized KL constraint problem, and augment it with a target function $s.$ Subtracting its expectation and applying the Gram-Schmidt process on $s$ as in \ref{rem:normalizng} is just equivalent to performing the classical \emph{linear regression} estimation on $s.$ For this reason, it is natural to interpret Theorem \ref{thm:7}, as well as Theorem \ref{thm:cumulants_M} below as a ''perturbed regression'' method, which generalizes linear regression to incorporate higher order moments.
\end{rem}

\subsection{Cumulant trees}
Let $\left(\Sigma,Q,\left\{ r_{i}\right\} _{i=1}^{k}\right)$ be
a normalized KL constraint problem. Now we will not have use for $r_0$; but as before, if the problem is augmented we will denote $r_{k+1}=s$.

\begin{defn}
For $j=1,\ldots,k+1$ \emph{a $j$-cumulant tree} is a rooted tree whose
edges are labeled by $\{r_{i}\}_{i=1}^{k+1}$ subject to the following
conditions:
\begin{enumerate}
\item Shape conditions: each vertex which is not a leaf has $\mbox{degree}\thinspace \geq3$.
\item Labeling conditions: The edge connected to $v_{out}$ is labeled by
$r_{j},$ all other edges (including edges connected to non-root leaves)
are labeled by $r_{1},\ldots,r_{k}$.
\end{enumerate}
A \emph{cumulant tree} is a $j$-cumulant tree for some $1\leq j\leq k+1$. We define the notions of isomorphism, automorphism and isomorphism types with similarly to the analogous notions for moment trees.
\end{defn}
\emph{The leaf multi-index} of $\labt$ is defined as before.
Again there are only a finite number of cumulant trees
with any given multi-index.
\begin{defn}\label{def:cumulant_tree_amp}
Let $\labt$ be a cumulant tree. Let $v$ be an inner vertex of $\labt$ of degree $k$ whose edges are labeled by $r_{i_{1}},\ldots,r_{i_{d}}$.
\begin{enumerate}
\item
The \emph{coupling value} $C_{v}$ of $v$ is
\[
C_{v}=(-1)^{d+1}\kappa(r_{i_{1}},\ldots,r_{i_{d}})
\]
where $\kappa(r_{i_{1}},\ldots,r_{i_{k}})$
is the \emph{joint cumulant} (cf. {Eq (\ref{eq:joint cumulants})}) of $r_{i_{1}},\ldots,r_{i_{k}}$.
\item
The \emph{amplitude} of $\labt$ is
\begin{equation}\label{eq:coupling_cum}
A_{\labt}:=(-1)^{|I_\Gamma|}\prod_{v\in V_{in}(\labt)}C_{v},
\end{equation}
or $-1$ when $V_{in}=\emptyset.$
\end{enumerate}
\end{defn}

\begin{thm}\label{thm:cumulants_L}
Given a normalized KL constraint problem, the coefficients of the perturbed exponential parameters, defined in \ref{eq:lambda_i}, are given by
\begin{equation}\label{eq:cumulants_L}
L_{j,I}=\sum \frac{A_{\labt}}{|Aut(\labt)|}\mbox{ for }1\leq j \leq k,
\end{equation}
where the sum runs over all isomorphism
types of $j$-cumulant trees $\labt$ with $I_{\labt}=I$.
\end{thm}

\begin{thm}\label{thm:cumulants_M}
Given an augmented normalized KL constraint problem, the coefficients of perturbed expectation, defined in \ref{eq:mu}, are given by
\[M_{I}=\sum\frac{A_{\labt}}{|Aut(\labt)|},\]
where the sum runs over all isomorphism types of $\left(k+1\right)$-cumulant trees $\labt$ with $I_{\labt}=I$.
\end{thm}
These theorems are proved in Subsection \ref{sec:333}.

\begin{rem}
The number of labeled trees is exponential where the base depends on the number of labels. Since cumulant trees have one less labels, working with them produces an exponential reduction in the complexity. In practice the improvement may be even more significant since cumulants tend to decay more rapidly than moments.
\end{rem}

\begin{rem}
Consider the equivalence relation on moment trees which is generated by contracting edges labeled by $0$. Equivalence classes of this relation correspond to cumulant trees, in an obvious way. This correspondence preserves the output edge label and the leaf multiindex, and
it is possible to show that the contribution of each cumulant tree in Eq (\ref{eq:cumulants_L}) is the sum of the contributions to Eq (\ref{eq:params-by-moments}) of the moment trees in the corresponding equivalence class. A similar statement holds for the pertrubed expectation.
\end{rem}

\subsection{Further Discussion and Future Directions}
The formulas presented in this paper give a first derivation of maximum entropy estimation under constraints. Although the derivation is fairly long, the formulas themselves are simple, and can be easily programmed. Of course, the formulas become longer and consist of more terms, the higher in perturbation order one goes. Writing an efficient algorithm for calculating the order by order contribution is a natural future direction. Another inviting future direction is to try and solve the maximum entropy problem described here in other models, for example Markovian models.

It is well known that the statistically correct solution to many scientific or algorithmic problems is given by the entropy maximizer. Because maximizing entropy is sometimes difficult, an alternative, less optimal solution is sometimes taken, such as optimizing second moment (e.g., linear regression). The most fascinating challenge arising from this work is to apply the formulas found in this paper to practical problems and to understand when they allow an improvement to the current solutions.

\section{Feynman Calculus and Proofs of the Theorems}
\subsection{Feynman Calculus: Perturbative Expression for Critical Point}\label{sec:feynman}
We now introduce the main technical tool of this paper - using Feynman calculus to express critical points of functions.

Let $\tau :\cc^{r + m}\to \cc$ be an analytic function. Consider $V = \cc^m$ as a vector space over $\cc$. We write $\tau^x(y)$ for the value of $\tau$ at $(x,y)$ and think of it as a family of functions $\tau^x : V \to\cc$ parameterized by $x \in \cc^r$. Henceforth, unless explicitly stated otherwise, functions between vector spaces are not necessarily linear.

We will assume that $\tau^0$ has a critical point at $y = 0$, which is non-degenerate (see below).
This will imply that for sufficiently small values of $x$ and $y$, $\tau^x$ obtains a unique critical point $y = crit(\tau^x)$, and the assignment $x\mapsto crit(\tau^x)$ is an analytic function. We will see that the coefficients of the Taylor expansion of $x\mapsto crit(\tau^x)$ around $x = 0$ are given in terms of summation over trees. The entire discussion can be viewed as a kind of effective version of the contraction mapping proof of the implicit function theorem. We now explain this in more detail.

Let $V^* = Hom(V,\cc)$ denote the dual vector space, consisting of $\cc$-linear maps from $V$ to $\cc$. The derivative of a function $h : V \to \cc$ is a function $\partial h : V \to V^*$. We define a function $f : \cc^{r} \times V \to V^*$ to be the partial derivative of $\tau^x$ in the $y$ direction,
\[f(x,y) = \partial \tau^x.\]
Thus $f(x,y)=0$ if and only if $y$ is a critical point for $\tau^x$. We assume that $\tau^0$ has a critical point at 0, so $f(0,0)=0$. We further assume that this critical point is \emph{non-degenerate}, which means
\[B :=\partial_y f|_{(x,y)=(0,0)} : V \to V^*\] is an invertible linear transformation.


Define $g : \cc^r \times V \to V$ by
\begin{equation}
g^x(y) = y - B^{-1} \circ f(x,y)
\end{equation}
Clearly, $f(x,y) =0$ iff $g^x(y)=y$. In particular, $g^0(0) = 0$.

Now consider \[\partial g : \cc^r \times V \to Hom(V,V).\]
We have \[(\partial g)(0,0) = \operatorname{id}_V - B^{-1}\circ B = 0.\]
By continuity, there exist open neighborhoods $0\in U\subseteq \cc^r, 0\in W\subseteq V$ such that \begin{enumerate}
\item For all $x\in U$, $g^x(W)\subseteq W$, and
\item for all $x \in U, y\in W$, $(\partial g)(x,y) \in Hom(V,V)$ is a (linear) contraction mapping with respect to the standard inner product on $V$.
\end{enumerate}
It follows that for all $x \in U$, $g^x : W \to W$ is a contraction mapping.

Now by the Banach fixed point theorem, for any $x\in U$ there exists a unique $g^x$-fixed point $y = crit(\tau^x) \in W$, \[g^x(crit(\tau^x))=crit(\tau^x).\]
Moreover, for any $x \in U$ the sequence $y_n^x$ defined by $y_0^x=0, y_1^x = g^x(0), y_2^x = g^x(g^x(0)),\ldots,y_{n+1}^x=g^x(y_n^x),$ converges to $crit(\tau^x)$.

By the standard contraction-mapping proof of the implicit function theorem \footnote{See for example the proof in Joel Feldman's lecture notes, \url{http://www.math.ubc.ca/~feldman/m425/impFnThm.pdf}. To apply the argument we need that $(\partial f)(x,y) \in Hom(V,V^*)$ is invertible for all $(x,y) \in U \times W$. Since $\partial f = B(\operatorname{id} - \partial g)$ this follows from our assumption that $(\partial g)(x,y) \in Hom(V,V)$ is contracting.}, one shows that the assignment $x \mapsto crit(\tau^x)$ is an analytic function $U \to W$.
Write the multivariate Taylor expansion of $\tau^x$ as
\[
\tau^{x}\left(y\right)=\widehat T_{0}^{x}+\widehat T_{1}^{x}y+\frac{1}{2!}(\widehat B + \widehat T_{2}^x)\left(y,y\right)+\sum_{l\geq3} \frac{1}{l!}\widehat T_l^x(y,\ldots,y).
\]
This expansion is unique provided we assume that $\widehat T_l^x \in Hom(V^{\otimes l},\cc) \simeq (V^*)^{\otimes l}$ is symmetric: $\widehat T_l^x\in Sym^l(V^*)\subset (V^*)^{\otimes l}$. \[\widehat B \in V^* \otimes V^*\] is obtained from $B \in Hom(V,V^*)$ via the natural isomorphism $Hom(V,V^*)\simeq V^*\otimes V^*$, so that $\widehat T_2^{x=0}=0$. Similarly, we define
\[T_l^x \in Hom(V^{\otimes (l-1)},V^*)\]
to be the image of $\widehat T_l^x$ under the natural isomorphism
$(V^*)^{\otimes l}\simeq
Hom(V^{\otimes l-1},V^*)$.
Unwinding definitions we see that

\begin{equation}
\label{eq:T-contraction}
g^x(y) = -B^{-1} T_1^x - B^{-1}T_2^x(y) - \sum_{l\geq3}\frac{1}{(l-1)!} B^{-1}T_l^x(y^{\otimes (l-1)})
\end{equation}
We now explain how $y_n^x := (g^x)^n(0)$ can be interpreted as a sum over rooted planar trees of height $\leq n$.

\begin{defn}
Let $\unlabt$ be a rooted tree.
The \emph{height} of $\unlabt,~h$, is the length of the longest simple path $(v_0,v_1,\ldots,v_h)$ in $\unlabt$ with $v_0 = v_{out}$.
A \emph{planar rooted tree} (or PRT) $\unlabt$ is a rooted tree with an additional specification of a cyclic order on the set of edges incident to $v$ for every $v \in V$.
\end{defn}

We now give a precise definition of the contribution $\cont^x(\unlabt)$ of a PRT to $y_n^x$. Some readers may be satisfied with the more informal definition in Remark \ref{rem:informal-cont} below.

Let $\unlabt$ be a planar rooted tree, with a distinguished leaf $v_{out}$. Fix some $v \in V$. We define an (ordered) tuple of planar rooted trees, \emph{the subtrees of $v$}, as follows. Remove the path from $v_{out}$ to $v$, including both endpoints. Each connected component of the resulting topological space corresponds to a subtree of $v$. To obtain this subtree, we append a leaf at the end of the half-open edge which was connected to $v$; this is the distinguished leaf of the subtree. The cyclic orders at the vertices are determined by those of $\unlabt$; the order on the set of subtrees of $v$ is determined by the cyclic order of the edges around $v$, starting from the (omitted) edge toward $v_{out}$.
If $\unlabt$ has height $\leq h$ then the subtrees of the root are of height $\leq h-1$ and completely specify the isomorphism type of $\unlabt$. This leads us to the following recursive construction.

For $h \geq 0$ we define the set $\pt^{\leq h}$ of \emph{representative planar rooted trees of height $\leq h$} by setting $\pt^{\leq 0} = \emptyset$ and $\pt^{\leq h} = Seq(\pt^{\leq h-1})$ for $h\geq 1$ where $Seq(A)=\bigcup_{k \in \zz_{\geq 0}} A^k$ denotes the set of all finite sequences $(a_1,...,a_k)$ with $k\geq 0$ and $a_i \in A$; by definition, $A^0 = \{()\}$ is the set containing only the empty sequence. The set $\pt := \bigcup_{h=0}^\infty \pt^{\leq h}$ is then a set of representatives for the isomorphism types of planar rooted trees.
%

Fix $x \in U \subset \cc^r$. For $l\geq 0$ define $S_l^x : V^{\otimes {l}} \to V$ by $S_l^x = B^{-1} T_{l+1}^x$. Now, define $\cont^x : \pt \to V$ recursively on the height. For $\pt^{\leq 0}$ define it to be $0.$
For $h\geq 1$ we define $\cont^x(a_1,...,a_l)$ for $(a_1,...,a_l) \in \pt^{\leq h}$ by
\[\cont^x((a_1,...,a_l)) = -\frac{1}{l!} S_l^x (\bigotimes_{i=1}^l \cont^x (a_i)) \]

\begin{rem}\label{rem:informal-cont}
Somewhat less formally, given a PRT $\unlabt$ we can compute $\cont^x(\unlabt)$ by placing $B^{-1} \in V\otimes V$ on the edges $E(\unlabt)$ and $\frac{(-1)}{(val(v)-1)!}\cdot \widehat T_{val(v)}^x \in (V^*)^{\otimes val(v)}$ on each vertex $v\in V(\unlabt)\setminus\{v_{out}\}$ of degree $val(v)$, and then using the incidence pairing to contract the tensors. The result is $\cont^x(\unlabt)\in V$.
\end{rem}
\begin{lem}\label{lem:fixed_point}
For $h \geq 0$ we have
\begin{equation}
\label{eq:lambda=trees}
y_{h}^x = \sum_{\unlabt \in \pt^{\leq h}} \cont^x(\unlabt)
\end{equation}
In particular,
\begin{equation}
\label{eq:lambda=trees_full}
crit(\tau^x) = \sum_{\unlabt \in \pt} \cont^x(\unlabt),
\end{equation}
where the right hand side converges for all $x \in U$.
\end{lem}
\begin{proof}
We show the result holds by induction on $h$.
For $h = 0$ $y_0^x = 0$ which by convention is equal to the empty sum. Now suppose we have established Equation (\ref{eq:lambda=trees}) for some $h$, let us show it holds for $h+1$.

By Equation (\ref{eq:T-contraction}) and Equation (\ref{eq:lambda=trees}) for $h$ we have:
\begin{eqnarray}
y_{h+1}^x &=& g^x(y^x_h) = \sum_{l\geq 0}\frac{(-1)}{l!}S_l^x((y_h^x)^{\otimes l}) =\\
&=& \sum_{l\geq 0} \frac{(-1)}{l!} S_l^x\left((\sum_{a \in \pt^{\leq h}} \cont^x(a))^{\otimes l}\right)
\end{eqnarray}
On the other hand we have

\begin{eqnarray}
\sum_{\Gamma \in \pt^{\leq (h+1)}} A^x(\Gamma) &=& \sum_{l\geq 0, (a_1,...,a_l) \in (\pt^{\leq h})^l} \frac{(-1)}{l!} S_l^x(\bigotimes A^x(a_i)) \\
\notag&=& \sum_{l\geq 0} \frac{(-1)}{l!} S_l^x \left((\sum_{a \in \pt^{\leq h}} A^x(a))^{\otimes l}\right)
\end{eqnarray}
which shows that Equation (\ref{eq:lambda=trees}) holds for $h+1$, and the proof of the first claim is complete. The second claim, Eq (\ref{eq:lambda=trees_full}), immediately follows.
\end{proof}

\subsubsection{Coordinate Expression for the Perturbed Critical Point}\label{subsec:coordinate-expression}
\begin{defn}
Let $L$ be a finite set.
\emph{An $L$-labeled rooted tree} (or \emph{labeled-RT}, if $L$ is clear from the context) is a rooted tree $\Gamma$ such that all the edges of $\Gamma$ are labeled by elements of $L$.
We will say an $L$-labeled rooted tree $\Gamma$ has \emph{output} $l_0 \in L$ if the edge connecting the distinguished leaf $v_{out}$ to the root is labeled $l_0$.
The automorphism group $Aut(\Gamma)$ of $\Gamma$ consists of maps $\psi : V(\Gamma) \to V(\Gamma)$ that fix the distinguished leaf $v_{out}$ and such that $u,v\in V(\Gamma)$ are connected by an edge labeled $l$ iff $\psi(u),\psi(v)$ are connected by an edge labeled $l$.
\end{defn}

Let $L$ be a set of size $m= \dim V.$ In this subsection, we assume that $B \in Sym^2(V^*)$ is the complexification of a positive definite pairing. In other words, there exists a basis $\{e_i\}_{i\in L}$ to $V$ such that
\begin{equation}
\label{eq:B=identity}B=\sum_{i\in L} e^*_i\otimes e^*_i,
\end{equation}
where $\{e_i^*\}$ be the basis of $V^*$ dual to $\{e_i\}.$

For $l \geq 1$ and $i_1,...,i_l \in L$ define $\theta_{i_1,\ldots,i_l} : \cc^r \to \cc$ by
\[\widehat T_l^x=\sum_{i_1,\ldots,i_l\in L}\theta_{i_1,\ldots,i_l}(x) e_{i_1}^*\otimes\cdots\otimes e_{i_l}^*.\]
The symmetry condition $\widehat T_l^x \in Sym(V^*) \subset (V^*)^{\otimes l}$ is equivalent to requiring $\theta_{i_1,\ldots,i_l}(x)$ to be invariant under any permutation of the indices $i_1,\ldots,i_l.$

\begin{defn}\label{def:amplitude_labeled}
Let $\Gamma$ be an $L$-labeled tree. For any internal vertex $v$ such that the edges incident to $v$ are labeled $i_1,\ldots,i_d$ write \[C_v^x=-\theta_{i_1,\ldots,i_d}(x).\]
We define \emph{the amplitude function} $A_\labt^x:\cc^r \to \cc$ of $\labt$ by
\begin{equation}\label{eq:general_amps}
A_\labt^x=\prod_{v\neq v_{out}}C_v^x,
\end{equation}
where the product is taken over all vertices except $v_{out}.$
\end{defn}


Write
\[crit(\tau^x) = \sum_{i\in L} crit(\tau^x)_i e_i\]
\begin{thm}\label{thm:feynman}
For $i\in L$ We have
\begin{equation}
crit(\tau^x)_i = \sum_\labt \frac{A_\labt^x}{|Aut(\labt)|}
\end{equation}
where $\labt$ ranges over a set of representatives for the isomorphism types of $L$-labeled rooted trees with output $i.$

In the special case when $T_2^x=0$ for all $x$ the sum is taken over trees where all vertices are either leaves or of degree at least $3.$
\end{thm}
The following definition will be useful for proving the theorem.
\begin{defn}
(a) An \emph{$L$-labeled planar rooted tree} $\tilde{\labt}$ is a planar rooted tree together with a labeling.\\
(b) The \emph{small amplitude} $\tilde{A}_{\tilde\Gamma}^x$ of an $L$-labeled planar rooted tree $\labt$ is defined by
\[\tilde{A}_{\tilde\Gamma}^x=\prod_{v\neq v_{out}}\frac{C_v^x}{(deg(v)-1)!},\]
where the product is taken over all vertices except $v_{out},$ and $deg(v)$ is the degree of $v.$
\end{defn}
\begin{proof}[Proof of Theorem \ref{thm:feynman}.]
The $i^{th}$ coordinate of the expression (\ref{eq:lambda=trees_full}) for $crit(\tau^x)$
is
\[
crit(\tau^x)_i = \sum_{\tilde\Gamma} \tilde{A}_{\tilde\Gamma}^x
\]
where $\tilde\Gamma$ ranges over the $L$-labeled planar rooted trees with output $i.$
There is a forgetful map $For$ from the set of labeled planar rooted trees to the set of labeled rooted trees, obtained by forgetting the cyclic orders.

Let $\Gamma$ be a labeled rooted tree. We claim that size of the fiber over $\Gamma$ of the forgetful map  is given by
\[|For^{-1}(\{\Gamma\})| = \frac{1}{|Aut(\Gamma)|}\prod (deg(v)-1)!.\]
where the product is taken over all vertices.

Indeed, fix some reference $\tilde{\Gamma}_0\in For^{-1}(\Gamma)$.  There is a group $G$ of order $|G| = \prod_{v \in V({\tilde{\Gamma}_0})} {(\deg(v) -1)!}$ which acts transitively on $For^{-1}(\Gamma)$ by changing the order of the subtrees. In fact $G$ can be constructed as the semi-direct product of the symmetric groups $\{S_{\deg(v) -1}\}_{v \in V({\tilde{\Gamma}_0})}.$ The stabilizer of $\tilde{\Gamma}_0$ is naturally identified with $Aut(\Gamma)$, so by the orbit-stabilizer theorem the size of the fiber is $ \frac{1}{|Aut(\Gamma)|}\prod (deg(v)-1)!$ as claimed.

The amplitude function $\tilde{A}^x_{\tilde\Gamma}$ is independent of the specific $\tilde{\Gamma}\in For^{-1}(\Gamma).$
By summing over all $For-$preimages we get,
\[
\sum_{\tilde\Gamma\in For^{-1}(\Gamma)} \tilde{A}_{\tilde\Gamma}^x=\frac{\prod (deg(v)-1)!}{|Aut(\Gamma)|}\prod \frac{C_v^x}{(deg(v)-1)!}=\frac{A^x_\Gamma}{|Aut(\Gamma)|}.
\]
As claimed.

The last claim follows from the observation that $C_v=0$ for any bivalent vertex.
\end{proof}

\subsection{Proofs of the main theorems}\label{sec:proofs}

Let $(\Sigma,Q,\{r_i\}_{i=1}^k)$ be a normalized KL-problem. Let $$\Delta'=\{{p_\sigma} \in \rr^\Sigma | (\forall \sigma ~ p_\sigma > 0) \wedge \sum_\sigma p_\sigma = 1\}$$ denote the \emph{open} simplex of probability distributions. We are looking for the probability distribution $P\in \Delta'$ that minimizes
\[D(P||Q) = \sum_{\sigma\in\Sigma}p_{\sigma}\log\left(\frac{p_{\sigma}}{q_{\sigma}}\right)\]
subject to the constraints $\ee_P r_i = \rho_i$. Using the assumption $q_\sigma > 0$, simple analysis shows that $D(P||Q)$ tends to $+\infty$ as we approach the boundary of $\Delta'$. To be precise, for every $M\in \rr$ there is a compact subset $K\subset \Delta'$ such that $D(P||Q) \geq M$ for all $P\in \Delta'-K$. It follows that $D(P||Q)$ obtains a minimum in $\Delta'$, which is unique since $D(P||Q)$ is a strictly convex function of $P$. Now apply Lagrange multipliers
\begin{align}\label{eq:KL-constraints}
D(P||Q) - &\sum_{i=1}^k \lambda_i (\ee_P (r_i) - \rho_i) = \\ \notag &\sum_{\sigma}Q_{\sigma}\frac{p_{\sigma}}{Q_{\sigma}}\left(-\log\frac{p_{\sigma}}{Q_{\sigma}}-\sum_{i=1}^{k}\lambda_{i}r_{i}(\sigma)-\lambda_{0}\right)+\sum_{i}\lambda_{i}\rho_{i}+\lambda_{0}.
\end{align}
Set $x_{\sigma}=\frac{p_{\sigma}}{Q_{\sigma}}.$ By requiring the vanishing of the partial derivatives with respect to $x_\sigma,$ we find that
\[
x_{\sigma}=\exp\left(-\sum_{i=1}^{k}\lambda_{i}r_{i}(\sigma)-\lambda_{0}-1\right).
\]
Plugging this into Equation \ref{eq:KL-constraints} we are looking for the minimum of
\begin{equation}\label{eq:tau_hat}
\hat{\tau}^\rho(\lambda) = \mathbb{E}_{Q}\exp\left(-\sum_{i=1}^{k}\lambda_{i}r_{i}(\sigma)-\lambda_{0}-1\right)+\sum\lambda_{i}\rho_{i}+\lambda_{0}
\end{equation}
We will compute $crit(\hat{\tau}^\rho),$ the critical point of $\hat{\tau}^\rho(\lambda),$ in two ways. First, by applying the Feynman calculus directly, we will express $crit(\hat{\tau}^\rho)$ as a sum over moment trees, and thus prove Theorem \ref{thm:taylor-exp-parameters}. Second, we can solve $\partial_{\lambda_0} T=0$ for $\lambda_0$ and only then apply the Feynman calculus. This way will lead to a sum over cumulant trees, and the formula of Theorem \ref{thm:cumulants_L}.

\subsubsection{Proof of moment tree theorems}\label{sec:exp_param}
\begin{proof}[Proof of Theorem \ref{thm:taylor-exp-parameters}]
Write $\lambda_i' = \lambda_i$ for $1\leq i\leq k$ and $\lambda_0'= \lambda_0 + 1.$
Define $\tau^\rho(\lambda')$ by analytically continuing $\hat{\tau}^\rho(\lambda_0'-1,\lambda_1',\ldots,\lambda_k')$ to $\cc^k \times V$ for $V = \cc^{k+1}$.
We assume that for $0 \leq i \leq k,~\lambda_i' \in V^*$ is the dual to the standard basis $\{e_i\}_{i=0}^k$ for $V.$ Since we assume $(\Sigma,Q,\{r_i\})$ is a normalized KL problem, the quadratic term of $\tau^\rho$ is
\[\frac{1}{2} \sum_{0\leq i,j \leq k} (\ee_Q r_i r_j)\lambda_i' \lambda_j' =\sum_{i=0}^k \lambda_i'^2\]
and Equation (\ref{eq:B=identity}) holds. Now apply Theorem \ref{thm:feynman}. In the notation of Subsection \ref{subsec:coordinate-expression}, with $x = \rho$ and $y=\lambda',$ we have
\begin{equation}\label{eq:moment_coefs}
\theta_{i_1,\ldots,i_l}^\rho = \begin{cases}
	  (-1)^l \ee_Q(\prod_{j=1}^l r_{i_j}) & \text{if } l \geq 3 \\
      \rho_{i_1} & \text{if } l = 1 \text{ and } 1\leq i_1 \leq k \\
      0 & \text{if } l = 1 \text{ and } i_1 = 0,~ \text{or}~ l=2.
   \end{cases}
\end{equation}
and
\begin{equation}\label{eq:moment_crit_val}
\lambda_i' = crit(\tau^\rho)_i = \sum_\labt \frac{A_\labt^\rho}{|Aut(\labt)|},
\end{equation}
where $\labt$ ranges over a set of representatives for the isomorphism types of $\{0,1,\ldots,k\}-$labeled rooted trees, with output $i,$ each vertex which is not a leaf is of degree at least $3,$ and, the edge of no leaf, other than the root, can be labeled $0.$ The last condition is a consequence of $\theta_0^\rho=0.$
Hence, $\labt$ ranges over a set of representatives for the isomorphism types of $i$-moment trees as defined in \ref{subsec:moment-trees}.
$A_\labt^\rho$ is the amplitude defined in Equation \ref{eq:general_amps}. But by Equation \ref{eq:moment_coefs}, $A_\labt^\rho$ is just $A_\labt\rho^{I_\labt},$ where $A_\labt$ is the amplitude defined in Definition \ref{def:moment_tree_amp}. Theorem \ref{thm:taylor-exp-parameters} is thus proved.
\end{proof}


\begin{lem}\label{cor:lambda_0}
The unique solution to $\frac{\partial}{\partial \lambda_0'} \tau^\rho =0$ is given by
\begin{equation}\label{eq:exponent_0}
\exp(\lambda_0')=\sum_{\sigma\in\Sigma}q_\sigma\exp(-\sum_{i=1}^k\lambda_ir_i(\sigma)).
\end{equation}
\end{lem}
\begin{proof}
By Equation \ref{eq:tau_hat} $\tau^\rho(\lambda')$ has the form
\begin{equation}\label{eq:tau}
\tau^\rho(\lambda') = \mathbb{E}_{Q}\exp\left(-\sum_{i=1}^{k}\lambda_{i}r_{i}(\sigma)-\lambda'_{0}\right)+\sum\lambda_{i}\rho_{i}+\lambda'_{0}-1
\end{equation}
A direct computation gives
$\mathbb{E}_{Q}\exp\left(-\sum_{i=1}^{k}\lambda_{i}r_{i}(\sigma)-\lambda'_{0}\right) = 1,$ or
\[
\mathbb{E}_{Q}\exp\left(-\sum_{i=1}^{k}\lambda_{i}r_{i}(\sigma)\right)=e^{\lambda_0'},\] as claimed.
\end{proof}

\begin{proof}[Proof of Theorem \ref{thm:7}]
We have
$$\mu(\lambda_{1},\dots,\lambda_{k})=\frac{\sum_{\sigma}s(\sigma)q_{\sigma}\exp\left(-\sum_{i=1}^{k}\lambda_{i}r_{i}\left(\sigma\right)\right)}
{\sum_{\sigma}q_{\sigma}\exp\left(-\sum_{i=1}^{k}\lambda_{i}r_{i}\left(\sigma\right)\right)}
$$
which by Lemma \ref{cor:lambda_0} can be rewritten as
\begin{equation}\label{eq:reduced_form_mu}
\mu(\lambda_{1},\dots,\lambda_{k})=\sum_{\sigma}s(\sigma)q_{\sigma}\exp\left(-\lambda_0'-\sum_{i=1}^{k}\lambda_{i}r_{i}\left(\sigma\right)\right)=\sum_{\sigma}s(\sigma)q_{\sigma}\exp\left(-\sum_{i=0}^{k}\lambda'_{i}r_{i}\left(\sigma\right)\right),
\end{equation}
where $r_0=1,$ as before. Expand $\exp(x)=\sum_{a\geq0}\frac{x^a}{a!}$ to obtain
\begin{align}\label{eq:mu_first_step}
\mu(\lambda_{1},\dots,\lambda_{k}) &= \sum_\sigma q_\sigma s(\sigma)\sum_{i_0,i_1,\ldots,i_k\geq 0}\prod\frac{\prod_{j=0}^k (-r_j(\sigma))^{i_j}\lambda'_j(\sigma)^{i_j}}{\prod_{j=0}^k i_j!}\\
\notag & =\sum_{i_0,i_1,\ldots,i_k\geq 0}\frac{(-1)^{\sum_{j=0}^k i_j}\ee_Q \left(s\prod_{j=0}^k r_j^{i_j}\right)\prod_{j=0}^k\lambda'_j(\sigma)^{i_j}}{\prod_{j=0}^k i_j!}.
\end{align}
We can rewrite the last equation as
\begin{equation}\label{eq:mu_basic_diagramatic_exp}
\mu(\lambda_{1},\dots,\lambda_{k}) = \sum_\Lambda\frac{C_v}{|Aut(\Lambda)|}\lambda^{'I_\Lambda}
\end{equation}
where the sum is taken over labeled rooted trees with a single non-leaf vertex $v$, where the output edge is labeled by $k+1,$ and the other edge-labels are taken from $\{0,1,\ldots,k\}.$ The coupling value $C_v$ is as in Definition \ref{def:moment_tree_amp}. By Equation \ref{eq:moment_crit_val}, we may write the critical value $\lambda'_i$ as a sum over labeled rooted trees. Substituting this into Equation \ref{eq:mu_basic_diagramatic_exp} can be interpreted as a sum over all possible ways of replacing the leaves of $\Lambda$ whose edge is labeled by $i\in\{0,\ldots,k\}$ by $i-$moment trees.
%
It follows that
\begin{equation}
\mu=\sum_{\Gamma}\frac{A_{\labt}\rho^{I_\Gamma}}{|Aut(\labt)|}
\end{equation}
where the sum is taken over $k+1-$moment trees $\labt$.
The $1/|Aut(\labt)|$ coefficient is obtained from the orbit-stabilizer theorem, by considering the action of $Aut(\Lambda) = \prod_{j=0}^{k} i_j !$ on the trees obtained after substitution.
\end{proof}

\subsubsection{Proofs of cumulant tree theorems}\label{sec:333}
\begin{proof}[Proof of Theorem \ref{thm:cumulants_L}]
Recall that the critical $\lambda_i',$ for $0\leq i\leq k,$ are those which minimize
\begin{equation}
\tau^\rho(\lambda') = \mathbb{E}_{Q}\exp\left(-\sum_{i=1}^{k}\lambda_{i}'r_{i}(\sigma)-\lambda'_{0}\right)+\sum\lambda_{i}'\rho_{i}+\lambda'_{0}-1
\end{equation}
of Equation \ref{eq:tau}.

In order to get an expression in terms of joint cumulants, subtitute $\lambda'_0$ such that $\frac{\partial}{\partial\lambda_0'} \tau^\rho =0$.
$\log(\sum_{\sigma\in\Sigma}q_\sigma\exp(-\sum_{i=1}^k\lambda_i' r_i(\sigma))),$
by Lemma \ref{cor:lambda_0}.
Denote by $T(\lambda_1,\ldots,\lambda_k)$ the result of this substitution (recall $\lambda_i = \lambda_i'$ for $i\neq 0$). Then
\[
T(\lambda_{1},\ldots,\lambda_{k})=e^{-\log E+1-1}E+\phi+\log E-1=\log E +\phi.
\]
where $E=\mathbb{E}_{Q}\exp\left(-\sum_{i=1}^{k}\lambda_{i}r_{i}(\sigma)\right)$
and $\phi=\sum_{i=1}^{k}\lambda_{i}\rho_{i}.$
By definition, $\log E$ is the generating function
for the joint cumulants. More precisely,
\begin{equation}
\label{eq:joint cumulants}
\log E=\sum_{n}\sum_{1\leq i_{1},\ldots,i_{n}\leq k}(-1)^{n}\frac{\kappa(r_{i_{1}},\ldots,r_{i_{n}})}{n!}\prod_{j=1}^{n}\lambda_{i_{j}}.
\end{equation}.
We claim that within the convergence domain there exists a unique critical value of $\lambda_i$. Indeed, $log E+\phi$ is convex since $E$ is a linear combination of exponents with non negative coefficients. By Theorem \ref{thm:feynman}, this unique critical value is given by Eq (\ref{eq:cumulants_L}). The shape conditions hold since the propogator is the quadratic tensor in $\log E+\phi,$ hence all vertices must either be of degree $1$ or degree at least $3.$
\end{proof}
\begin{proof}[Proof of Theorem \ref{thm:cumulants_M}]
As in the proof of Theorem \ref{thm:7} the strategy will be to substitute the expressions of $\lambda_i$ in terms of cumulants in
\[
\mu(\lambda_{1},\dots,\lambda_{k})=\frac{\sum_{\sigma}s(\sigma)q_{\sigma}\exp\left(-\sum_{i=1}^{k}\lambda_{i}r_{i}\left(\sigma\right)\right)}
{\sum_{\sigma}q_{\sigma}\exp\left(-\sum_{i=1}^{k}\lambda_{i}r_{i}\left(\sigma\right)\right)}
\]
Introduce a new variable $\alpha$ and write
$$F(\lambda,\alpha)=F(\lambda_1,\ldots,\lambda_k,\alpha)=\log\left(\ee_Q\exp\left(\alpha s-\sum_{i=1}^{k}\lambda_{i}r_{i}\left(\sigma\right)\right)\right).$$
Then $$\mu(\lambda_1,\ldots,\lambda_k)=\frac{\partial}{\partial\alpha}F(\lambda_1,\ldots,\lambda_k,\alpha)|_{\alpha=0}.$$
Now by the definition of cumulants again,
$$F(\lambda,\alpha) = \sum_{1\leq i_{1},\ldots,i_{n}\leq k+1}(-1)^{n-n_{k+1}}\frac{\kappa(r_{i_{1}},\ldots,r_{i_{n}})}{n!}\prod_{j=1}^{n}\lambda_{i_{j}},$$
where $n_{k+1}$ is the number of indices $i_j=k+1$ and $\lambda_{k+1}=\alpha.$
Hence the partial derivative in $\alpha=0$ is just
$$\mu(\lambda)=\sum_{1\leq i_{1},\ldots,i_{n}\leq k}(-1)^{n}\frac{\kappa(r_{k+1},r_{i_{1}},\ldots,r_{i_{n}})}{n!}\prod_{j=1}^{n}\lambda_{i_{j}}.$$

As in the proof of Theorem \ref{thm:7}, $\mu$ can thus be written as
\begin{equation}\label{eq:mu_basic_diagramatic_exp_cumulant}
\mu(\lambda_{1},\dots,\lambda_{k}) = \sum_\Lambda\frac{C_v}{|Aut(\Lambda)|}\lambda^{I_\Lambda}
\end{equation}
where the summation is taken over labeled rooted trees with a single non-leaf vertex $v,$
~output $k+1,$ and the other labels are in $\{1,\ldots,k\}.$ The coupling value is as in Definition \ref{def:cumulant_tree_amp}. The proof now ends by substituting the expressions for $\lambda_i,~1\leq i \leq k$ in terms of cumulants, just as in the proof of Theorem \ref{thm:7}.
\end{proof}
\section*{acknowledgements}
We thank O. Bozo, B. Gomberg, R.S. Melzer, A. Moscovitch-Eiger, R. Schweiger, A. Solomon and D. Zernik for discussions related to the work presented here.

R.T. was partially supported by Dr. Max R\"{o}ssler, the Walter Haefner Foundation and the ETH Zurich Foundation.

\appendix
\section{\label{sub:pf of analyticity}Proof of Lemma \ref{lem:analytic}}
Recall that by Lemma \ref{cor:lambda_0}
\[
\log\left(\sum_{\sigma}q_{\sigma}\exp\left(-\sum_{i=1}^{k}\lambda_{i}r_{i}\left(\sigma\right)\right)\right)=1+\lambda_{0}=1+\lambda_{0}(\lambda_{1},\dots,\lambda_{k}).
\]
Hence $\lambda_{0}$ is an analytic function of $\lambda_{1},\dots,\lambda_{k}$
around $\lambda_{1}=\dots=\lambda_{k}=0$. Now,
\[
\rho_{i}(\lambda_{1},\dots,\lambda_{k})=\sum_{\sigma}r_{i}(\sigma)q_{\sigma}\exp\left(-1-\lambda_{0}(\lambda_{1},\dots,\lambda_{k})-\sum_{l=1}^{k}\lambda_{l}r_{l}\left(\sigma\right)\right)=
\]
\[
=\frac{\sum_{\sigma}r_{i}(\sigma)q_{\sigma}\exp\left(-\sum_{l=1}^{k}\lambda_{l}r_{l}\left(\sigma\right)\right)}{\exp(1+\lambda_{0}(\lambda_{1},\dots,\lambda_{k}))}=\frac{\sum_{\sigma}r_{i}(\sigma)q_{\sigma}\exp\left(-\sum_{l=1}^{k}\lambda_{l}r_{l}\left(\sigma\right)\right)}{\sum_{\sigma}q_{\sigma}\exp\left(-\sum_{l=1}^{k}\lambda_{l}r_{l}\left(\sigma\right)\right)}
\]
So that $\rho_{i}(\lambda_{1},\dots,\lambda_{k})$ is an analytic function of $\lambda_{1},\dots,\lambda_{k}$

The proof that $\lambda_{i}=\lambda_{i}\left(\rho_{1},\ldots,\rho_{k}\right)$
is an analytic function of $\rho_{1},\dots,\rho_{k}$ around
\[
\lambda_{1}=\dots=\lambda_{k}=\rho_{1}=\dots=\rho_{k}=0,
\]
uses the analytic inverse function theorem. It is enough to show that the Jacobian
$\frac{\partial(\rho_{1},\dots,\rho_{k})}{\partial(\lambda_{1},\dots,\lambda_{k})}$
is invertible for$\lambda_{1}=\dots=\lambda_{k}=0.$

But
\begin{align}
\frac{\partial\rho_{i}}{\partial\lambda_{j}}=&
\frac{\left(\sum_{\sigma}r_{i}(\sigma)r_{j}(\sigma)q_{\sigma}\exp\left(-\sum_{l=1}^{k}\lambda_{l}r_{l}\left(\sigma\right)\right)\right)\left(\sum_{\sigma}q_{\sigma}\exp\left(-\sum_{l=1}^{k}\lambda_{l}r_{l}\left(\sigma\right)\right)\right)}
{\left(\sum_{\sigma}q_{\sigma}\exp\left(-\sum_{l=1}^{k}\lambda_{l}r_{l}\left(\sigma\right)\right)\right)^{2}}-\\
\notag&
-\frac{\left(\sum_{\sigma}r_{l}(\sigma)q_{\sigma}\exp\left(-\sum_{l=1}^{k}\lambda_{l}r_{l}\left(\sigma\right)\right)\right)\left(\sum_{\sigma}r_{j}(\sigma)q_{\sigma}\exp\left(-\sum_{l=1}^{k}\lambda_{l}r_{l}\left(\sigma\right)\right)\right)}{{\left(\sum_{\sigma}q_{\sigma}\exp\left(-\sum_{l=1}^{k}\lambda_{l}r_{l}\left(\sigma\right)\right)\right)^{2}}}.
\end{align}
Evaluation at $\lambda_{1}=\dots=\lambda_{k}=0$ gives
\[
\frac{\partial\rho_{i}}{\partial\lambda_{j}}\left|_{\lambda_{1}=\cdots=\lambda_{k}=0}\right.=\frac{\left(\sum_{\sigma}r_{i}(\sigma)r_{j}(\sigma)q_{\sigma}\right)\left(\sum_{\sigma}q_{\sigma}\right)-\left(\sum_{\sigma}r_{i}(\sigma)q_{\sigma}\right)\left(\sum_{\sigma}r_{j}(\sigma)q_{\sigma}\right)}{\left(\sum_{\sigma}q_{\sigma}\right)^{2}}=
\]
\[
=\frac{\mathbb{E}_{Q}(r_{i}r_{j})\cdot1-\mathbb{E}_{Q}(r_{i})\mathbb{E}_{Q}(r_{j})}{1^{2}}=\mathrm{Cov}_{Q}(r_{i},r_{j}).
\]
By assumption the KL constraint problem is normalized, hence
\[
\frac{\partial\rho_{i}}{\partial\lambda_{j}}\left|_{\lambda_{1}=\cdots=\lambda_{k}=0}\right.=\mathrm{Cov}_{Q}(r_{i},r_{j})=\delta_{i,j}.
\]
The Jacobian $\frac{\partial(\rho_{1},\dots,\rho_{k})}{\partial(\lambda_{1},\dots,\lambda_{k})}=I_{k}$
is thus invertible.
\qed


\bibliographystyle{plain}

\end{document}